\documentclass[a4paper, 11pt, headings = small, abstract]{scrartcl}

\usepackage[english]{babel}
\usepackage[utf8]{inputenc}
\usepackage{geometry}

\usepackage[semibold]{libertine}
\usepackage[upint,amsthm]{libertinust1math}

\usepackage{amsmath}
\usepackage{amsfonts}
\usepackage{amssymb}
\usepackage{amsthm}
\usepackage{mathtools}
\usepackage{paralist}
\usepackage{comment}

\usepackage[authoryear]{natbib}
\usepackage[dvipsnames]{xcolor}
\usepackage[colorlinks=true,linkcolor=blue,citecolor=blue,pdfborder={0 0 0}]{hyperref}
\usepackage{orcidlink}

\usepackage{dsfont}
\usepackage{bm}

\usepackage{graphicx}
\usepackage{enumerate}

\usepackage{tikz}
\usetikzlibrary{positioning}
\usetikzlibrary{backgrounds}

\newtheorem{theorem}{Theorem}[section]
\newtheorem{proposition}[theorem]{Proposition}
\newtheorem{lemma}[theorem]{Lemma}
\newtheorem{corollary}[theorem]{Corollary}

\theoremstyle{definition}

\theoremstyle{remark}
\newtheorem{remark}[theorem]{Remark}

\allowdisplaybreaks[4]

\numberwithin{equation}{section}

\newcommand{\R}{\mathbb{R}}

\newcommand{\N}{\mathbb{N}}
\newcommand{\eps}{\varepsilon}

\newcommand{\weak}{\rightsquigarrow}

\newcommand{\Exp}{\operatorname{E}}
\newcommand{\Var}{\operatorname{Var}}

\newcommand{\diff}{\,\mathrm{d}}

\begin{document}

\title{\fontsize{16}{19} On the lack of weak continuity of Chatterjee's correlation coefficient}

\author{
Axel B\"ucher\thanks{Ruhr-Universität Bochum, Fakultät für Mathematik. Email: \href{mailto:axel.buecher@rub.de}{axel.buecher@rub.de}}~\orcidlink{0000-0002-1947-1617}
\and
Holger Dette\thanks{Ruhr-Universität Bochum, Fakultät für Mathematik. Email: \href{mailto:holger.dette@rub.de}{holger.dette@rub.de}}
~\orcidlink{0000-0001-7048-474X}
}

\date{\today}

\maketitle

\begin{abstract} 
Chatterjee's correlation coefficient has recently been proposed as a new association measure for bivariate random vectors that satisfies a number of desirable properties. Among these properties is the feature that the coefficient equals one if and only if one of the variables is a measurable function of the other. As already observed in Mikusinski, Sherwood and Taylor (Stochastica, 13(1):61–74, 1992), this property implies that Chatterjee's coefficient is not continuous with respect to weak convergence. We discuss a number of negative consequences for statistical inference. In particular, we show that asymptotic tests for stochastic independence based on Chatterjee's empirical correlation coefficient, or boosted versions thereof, have trivial power against certain alternatives for which the population coefficient is one.
\end{abstract}


\section{Introduction}

Measuring association between real-valued random variables is an important and historical branch of stochastics that is still actively researched today (see \citealp{Cha24} for a recent review).
The scientific discussion of new such measures or previously unexplored aspects of existing measures is usually based on a set of desirable axioms that a meaningful association measure should fulfill.
The first set of such axioms was formulated in \cite{Ren59} and reads as follows: for two real-valued random variables $X$ and $Y$ defined on a common probability space, the association measure $R=R(X,Y)$ is supposed to satisfy:
\begin{enumerate}[(A)]
\item $R(X,Y)$ is defined for any $X$ and $Y$.
\item $R(X,Y)=R(Y,X)$.
\item $R(X,Y) \in [0,1]$.
\item $R(X,Y) = 0$ if and only if $X$ and $Y$ are independent.
\item $R(X,Y) = 1$ if either $X=f(Y)$ a.s.\ or $Y=f(X)$ a.s.\ for some measurable function $f$.
\item $R$ is invariant under one-to-one mappings of the margins.
\item $R(X,Y)=|\varrho|$, provided $(X,Y)$ is jointly normal with correlation $\varrho$.
\end{enumerate}
We refer to \cite{SchWol81} for some early discussions, extensions and alternatives.

Recent years have witnessed a growing interest in association measures that concentrate on the following non-symmetric `if and only if' version of (E):
\begin{enumerate}
\item[(E')] $R(X,Y) = 1$ if and only if $Y=f(X)$ a.s.\ for some measurable function $f$.
\end{enumerate}
In fact, it has been shown that (C), (D) and (E') are satisfied if $R(X,Y) =\xi(X,Y)$ denotes the population version of the famous Chatterjee correlation coefficient \citep{Cha21} 
\begin{equation}
\label{det1}
\xi(X,Y) := \frac{\int \Var( \Exp[ \bm 1(Y \ge y) \mid X]) \diff P_Y(t) }{ \int \Var(\bm 1(Y \ge y)) \diff P_Y(y)},
\end{equation}
defined for non-constant $Y$ (and hence not satisfying (A)) and
previously introduced in \cite{DetSto13} for the case where both $X$ and $Y$ have a continuous cumulative distribution function (cdf), see Theorem 1 in \cite{Cha21}.
Note in passing that the symmetrized version $R^*(X,Y)=\max(\xi(X,Y), \xi(Y,X))$ satisfies `if and only if' in (E), and, of course, (B).

A  couple of additional or alternative axioms to Renyi's  list have been proposed in \cite{SchWol81}, and the one that is most important for this note is the following:
\begin{enumerate}
\item[(H)] The measure $R$ is continuous with respect to weak convergence, that is, if $(X_k,Y_k)$ is a sequence of random variables weakly converging to $(X,Y)$, then $\lim_{k \to \infty} R(X_k, Y_k) = R(X,Y)$. 
\end{enumerate}
The authors believe that this axiom is very natural, in particular when it comes to statistics, which, on a high level, is ultimately based solely on the fact that empirical measures weakly converge to the population measure almost surely if the sample size converges to infinity \citep{Var58}. However, it has already been observed in \cite{MikShe92} that it is impossible for an association measure that satisfies (A)-(D) to satisfy both (E') and (H); see also Corollary~\ref{cor:R} below for a slightly sharper statement. In fact, the following corollary, proven in Section~\ref{sec:main}, illustrates this pitfall very clearly for the Chatterjee correlation coefficient.

\begin{corollary} \label{cor:Xi}
For any pair of independent random variables $(X, Y)$ with continuous marginal cdfs and any $\xi_0 \in [0,1]$, there exists a sequence of random variables $(X_k, Y_k)$ weakly converging to $(X,Y)$ for $k\to\infty$ such that Chatterjee's correlation coefficient satisfies $\xi(X_k, Y_k)=\xi_0$ for all $k$.
\end{corollary}

The result essentially shows that Chatterjee's correlation coefficient is intrinsically hard to interpret. Arbitrary large values of $\xi(X,Y)$ may be obtained for random variables which are `almost independent' in the sense that their induced law is close to that of the associated product measure of its margins, with `closeness' measured in terms of any metric that metrizes weak convergence. The main purpose of this note is to shed further light on this issue, with a particular view on statistical applications such as independence testing. 

More precisely, regarding independence testing, we show in Section~\ref{sec:independence} that any sufficiently regular test for the null hypothesis of independence has trivial power, even in a classical Neymann-Pearson type situation of testing $H_0: \xi(X,Y)=0$ against $H_1:\xi(X,Y)=1$. In an asymptotic framework this means that we can find a sequence of alternatives with trivial power which can be regarded both `local' in the sense that it converges weakly to independence, and `non-local' in the sense that we have $\xi=1$ for each alternative in the sequence. We thereby complement recent results from \cite{ShiDrtHan22} and \cite{CaoBic20}, who found trivial power for Chatterjee's asymptotic test \citep{Cha21} along certain local alternatives that converge to the null at a non-optimal rate.

A second statistical application, discussed in Section~\ref{sec:confint}, concerns the construction of (asymptotic) confidence intervals. We show that any sufficiently regular confidence interval of level $\alpha$ must be unreasonably large, with probability at least $1-2\alpha$, for any random sample from a pair of random variables with continuous margins. Likewise, meaningful uniform asymptotic confidence intervals cannot exist -- we want to stress though that no assertion is made on the more common notion of (pointwise) asymptotic confidence intervals; such intervals may well exist in a meaningful way.

Our results apply in particular to Chatterjee's empirical correlation coefficient \citep{Cha21} and versions thereof. The latter has undergone intensive research in the recent years, and many useful properties like invariance under monotone transformations of the data, low computational complexity, consistency, asymptotic normality \citep{LinHan22} or subsampling consistency \citep{DetKro24} have been established; see also \cite{Aza21, linfan2022,
LinHan24, Shi24}  for extensions and further results. Despite these impressive advances, our results show that some care is necessary when interpreting the coefficient and conducting statistical inference.

\section{Shuffles of the Min and measures of perfect dependence}
\label{sec:main}

Recall that a bivariate copula $C:[0,1]^2 \to [0,1]$ is the restriction of a bivariate cdf with standard uniform margins to the unit square \citep{Nel06}. A specific copula is given by the Fr\'echet-Hoeffding upper bound $C_{\min}(u,v) = \min(u,v)$, which corresponds to the uniform distribution on the main diagonal $D=\{(u,v) \in [0,1]^2: u=v\}$ of the unit square.

A shuffle of Min (i.e., of $C_{\min})$ is a specific bivariate copula which is, informally, constructed as follows: put mass 1 on $D$, then cut  $[0,1]^2$ vertically into a finite number of strips, shuffle the strips with perhaps some of them flipped around their vertical axes of symmetry, and then reassemble the strips to form the square again. The resulting mass distribution has uniform margins, and hence defines a unique copula. We refer to \cite{MikShe92} or Section 3.2.3 in \cite{Nel06} for details. The important aspect for this note is that the support of a shuffle of Min is the graph of a bijective function, and therefore it is is the cdf of a random vector $(X,Y)$ for which $Y$ is a bijective function of $X$. In particular, if an association measure $R$ satisfies (E'), we have $R(X,Y)=1$ if $(X,Y)$ is distributed according to a shuffle of Min. 

A well-known result from copula theory states that shuffles of Min are dense in the space of copulas with respect to the topology of uniform convergence.

\begin{theorem}[\citealp{KimSam78, MikShe92}] \label{theo:shuffle}
Let $C$ denote an arbitrary copula on $[0,1]^2$. There exists a sequence of copulas $(C_k)_k$ converging to $C$ in uniform distance such that each $C_k$ has the following property: if $(X_k,Y_k) \sim C_k$, then $Y_k=f_k(X_k)$ almost surely for some measurable function $f_k$. More precisely, each $C_k$ is a shuffle of Min.
\end{theorem}

A proof for the independence copula $C(u,v)=uv$ can be found in the monograph  \cite{Nel06}, see his Theorem 3.2.2. The result immediately implies the following fact from the introduction that has already been mentioned in \cite{MikShe92}.

\begin{corollary} \label{cor:R}
Suppose that the association measure $R$ satisfies (D) and (E') and is well-defined for all random vectors with standard uniform margins. Then $R$ does not satisfy (H). 
\end{corollary}

A more precise statement can be obtained for the special case of Chatterjee's coefficient, that is, for $R(X,Y)=\xi(X,Y)$. Corollary~\ref{cor:Xi} from the introduction is a straightforward consequence.

\begin{proposition} \label{prop:Xi}
For any pair of random variables $(X, Y)$ with continuous marginal cdfs and any $\xi_0 \in [\xi(X,Y),1]$, there exists a sequence of random variables $(X_k, Y_k)$ (with continuous marginal cdfs) weakly converging to $(X,Y)$ for $k\to\infty$ such that $\xi(X_k, Y_k)=\xi_0$ for all $k$.
\end{proposition}

\begin{proof}[Proof of Proposition~\ref{prop:Xi}]
It is instructive to start with the case where $X$ and $Y$ are independent (i.e., Corollary~\ref{cor:Xi}).
Since $\xi$ is invariant under strictly monotone transformations of the margins, it is sufficient to show the statement for $(X,Y) \sim \Pi$, where $\Pi(u,v) = uv$ denotes the independence copula.

Fix $\xi_0 \in [0,1]$. For $\alpha \in [0,1]$, 
let $(X_k, Y_k) \sim D_k:=\alpha C_k + (1-\alpha) \Pi$, where $C_k$ is the copula from Theorem~\ref{theo:shuffle} applied to $C=\Pi$.
Clearly, $D_k$ is a copula that converges to $\Pi$ in uniform distance, for any choice of $\alpha$, and hence $(X_k, Y_k)$ converges weakly to $(X,Y)$. Therefore, it remains to show that we can choose $\alpha$ in such a way that $\xi(X_k, Y_k)=\xi_0$. For that purpose, let  $\| f\|_2$ denote the $L^2$-norm of a real-valued function on $[0,1]^2$. Note that the partial derivative $\partial_1 C (u,v) = \frac{\partial}{\partial u} C(u,v)$ exists for any copula $C$ and for almost any $(u,v) \in [0,1]^2$, that $\xi(X_k, Y_k) = 6 \| \partial_1 D_k \|_2^2 -2$ \citep{DetSto13} and that
\[
\| \partial_1 D_k \|_2^2
=
\alpha^2 \| \partial_1 C_k \|_2^2 + (1-\alpha)^2 \| \partial_1 \Pi \|_2^2 + 2 \alpha(1-\alpha) \int_0^1 \int_0^1 \partial_1 C_k(u,v) \partial_1 \Pi(u,v) \diff u \diff v.
\]
Here, for instance by Lemma 1 in \cite{DetSto13}, $\| \partial_1 C_k \|_2^2=1/2$  and $\| \partial_1 \Pi \|_2^2=1/3$. Moreover, since $\partial_1 \Pi(u,v) = v$, 
\[
\int_0^1 \int_0^1 \partial_1 C_k(u,v) \partial_1 \Pi(u,v) \diff u \diff v
=
\int_0^1 v \int_0^1 \partial_1 C_k(u,v) \diff u \diff v
= 
\int_0^1 v^2 \diff v = \frac13,
\]
where we have used that $C_k(1,v)=v$ and $C_k(0,v)=0$. As a consequence, 
\[
\| \partial_1 D_k \|_2^2 = \frac{\alpha^2}2 + \frac{(1-\alpha)^2}{3} + \frac{2\alpha(1-\alpha)}{3} = \frac{\alpha^2+2}6.
\]
Hence, the claim follows by choosing $\alpha=\sqrt{\xi_0}$.

Next, suppose that $(X,Y)$ is an arbitrary random vector with continuous margins. Without loss of generality, we may assume that the margins are standard uniform, so let 
$(X,Y)\sim C^*$ with $C^*$ an arbitrary copula. Define $\xi^*=\xi(X,Y)$ and fix $\xi_0 \in [\xi^*,1]$. 
Start by choosing $C_k$ converging to $C^*$ in sup-norm as in Theorem~\ref{theo:shuffle}, and define $D_k = \alpha_k C_k + (1-\alpha_k)C^*$, which also converges to $C^*$ in sup-norm. Following the above argumentation, we get
\begin{align}
\label{eq:fk}
f_k(\alpha_k) := \| \partial_1 D_k \|_2^2 = \frac{\alpha_k^2}2 + (1-\alpha_k)^2 \frac{\xi^*+2}6 + 2\alpha_k(1-\alpha_k) \int_0^1 \int_0^1 \partial_1 C_k(u,v) \partial_1 C^*(u,v) \diff u \diff v. 
\end{align}
Here, $f_k(0)=(\xi^*+2)/6$ and $f_k(1) = 1/2$ and $f_k$ is continuous on $[0,1]$. As a consequence, by the intermediate value theorem and noting that $(\xi_0+2)/6 \in [(\xi^*+2)/6, 1/2]$, we may choose $\alpha_k$ in such a way that $f_k(\alpha_k) = (\xi_0+2)/6$, which implies that $\xi(D_k) = \xi_0$ is constant in $k$.
\end{proof}

\begin{remark}
Some additional thoughts reveal that the range of the function $f_k$ in \eqref{eq:fk} is exactly the interval $[(\xi^*+2)/6, 1/2]$. In that sense, the proposition's assertion cannot be strengthened, at least not by the current proof.
\end{remark}

\section{Consequences for testing independence}
\label{sec:independence}

Suppose $(X_1, Y_1), \dots, (X_n, Y_n)$ is a random sample of $(X,Y)$, where the distribution of $Y$ is non-degenerate. Throughout, we are interested in testing hypotheses regarding Chatterjee's correlation coefficient $\xi(X,Y)$, for instance, for the simple null hypothesis $H_0: \xi(X,Y) =0$ of independence between $X$ and $Y$. 

Analyzing respective tests formally requires a statistical model for the data generating process, along with a precise definition of the null and alternative hypotheses.
For that purpose, it is instructive to identify $\xi(X,Y)$ with $\xi(P)$, where $P$ is the joint distribution of $(X,Y)$. Note that Chatterjee's coefficient may then be regarded as a functional $\xi:\mathcal P \to [0,1]$ defined on the set $\mathcal P$ of probability distributions on $\R^2$ whose second marginal distribution is non-degenerate. 

Subsequently, let $\mathcal P_*\subset \mathcal P$ denote a set of candidate models for the distribution of $(X,Y)$; for instance, the set of bivariate probability measures whose margins have a continuous cdf. A more precise formulation of the testing problem at the beginning of this section is then as follows: based on one observation from the product experiment defined by the sample space $\mathcal X_n = (\R^2)^n$ equipped with the distribution family $\mathcal P_*^{(n)} = \{ P^{\otimes n}: P \in \mathcal P_*\}$, with $P^{\otimes n}$ the $n$-fold product measure of $P$, decide between the hypothesis
\begin{align} \label{eq:hypo}
H_0: P \in \mathcal P_*(A_0)~~\text{ vs. }~~ H_1: P \in \mathcal P_*(A_1) ,
\end{align}
where $A_0$ and $A_1$ are disjoint subsets of $[0,1]$ and where $\mathcal P_*(A) = \{ P \in \mathcal P_* : \xi(P) \in A\}$. 

Heuristically, we should expect the largest power in a Neymann-Pearson-type situation where we know in advance that only the most extreme cases $\xi(P)=0$ or $\xi(P)=1$ are possible (i.e., $\mathcal P_* \subset \{ P \in \mathcal P: \xi(P) \in \{0,1\}\}$), and where we want to decide between the hypotheses 
\begin{align} \label{hol1} 
H_0: \xi(P)=0~~\text{ vs. }~~H_1: \xi(P) = 1;
\end{align}
note that this corresponds to $A_0=\{0\}$ and $A_1=\{1\}$ in \eqref{eq:hypo}. 
A possible test statistic, as proposed by \cite{Cha21}, is $\varphi_{\xi,n} = \bm 1(\sqrt {5n/2} |\hat \xi_n| \ge u_{1-\alpha/2})$ with the empirical Chatterjee coefficient $\hat \xi_n$ and $u_{1-\alpha}$ the $(1-\alpha/2)$-quantile of the standard normal distribution; this test has asymptotic level $\alpha$ by the limit results in the last-named paper. Alternatively, under the additional assumption of marginal continuity, we may take  $\varphi_{\xi,n, \text{exact}} = \bm 1(|\hat \xi_n| \ge u_{n,1-\alpha})$ with $u_{n,1-\alpha}$ chosen as the $1-\alpha$ quantile of the (finite-sample) null distribution of $|\hat \xi_n|$. Note that the latter is pivotal under continuity of the margins and can hence be approximated to an arbitrary precision based on Monte Carlo simulation; the resulting test then has exact (non-asymptotic) level $\alpha$. 

In a nutshell, the results in the following lemma imply that, for any sufficiently regular test $\varphi_n$ of size $\alpha$, there exists an alternative $P_1$ with $\xi(P_1) = 1$ such that the power of $\varphi_n$ is arbitrary close to $\alpha$. In other words, any such test is not even able to separate the `most opposite hypotheses possible' from \eqref{hol1}. Moreover, the issue does not disappear asymptotically by considering asymptotic sizes; any sufficiently regular sequence of tests cannot be uniformly consistent against $H_1$. Note that similar results were derived in \cite{For09}, Theorem 1, relying on arguments from \cite{Rom04}.

\begin{lemma}
\label{lem:indeptest}
Let $\mathcal P_* \subset \{ P \in \mathcal P: \xi(P) \in \{0,1\}\}$ denote a set of candidate models that includes all continuous bivariate distributions with $\xi(P) \in \{0,1\}$.
For $n\in \N$, suppose that $\varphi_n$ is a test for the hypotheses in \eqref{hol1} 
that is of the form $\varphi_n = \bm 1( T_n \ge 0)$ where $T_n: (\R^2)^n \to \R$ is continuous on the  set of all sample vectors in $(\R^2)^n$ for which there are no ties in any marginal sample. Then, 
\begin{align} \label{eq:test}
\inf_{P \in \mathcal P_*(\{1\})} \Exp_{P^{\otimes n}}[\varphi_n] 
\le 
\Exp_{P_0^{\otimes n}}[\varphi_n].
\end{align}
where $P_0=\mathrm{Uniform} ([0,1]^2)$.
\end{lemma}

\begin{proof}
As a consequence of Proposition~\ref{prop:Xi}, there exists a sequence of probability measures $P_k \in \mathcal P_*$ such that $P_k \in \mathcal P_*(\{1\})$ and $P_k \weak P_0 \in \mathcal P_*(\{0\})$ for $k\to \infty$. In view of our assumption on $T_n$, the continuous mapping theorem implies that $P_k^{\otimes n} \circ T_n^{-1} \weak P_0^{\otimes n} \circ T_n^{-1}$ for $k \to \infty$.
Hence, for any given $\eps>0$, the Portmanteau theorem implies
\[
\Exp_{P_k^{\otimes n}}[\varphi_n] 
=
P_k^{\otimes n}(T_n \ge 0) 
\le 
P_0^{\otimes n}(T_n \ge 0) + \eps
=
\Exp_{P_0^{\otimes n}}[\varphi_n] + \eps
\]
by choosing $k$ sufficiently large.
This implies the assertion since $P_k\in\mathcal P_*(\{1\})$ and since $\eps>0$ was arbitrary.
\end{proof}

\begin{remark}~
\begin{compactenum}[(i)]

\item As a consequence of Lemma~\ref{lem:indeptest}, if $\varphi_n$ has (pointwise) level $\alpha$ (in fact, level $\alpha$ for continuous margins is sufficient), we obtain that $\inf_{P \in \mathcal P_*(\{1\})} \Exp_{P^{\otimes n}}[\varphi_n] \le \alpha $. Likewise, if $\varphi_n$ is a sequence of tests as in Lemma~\ref{lem:indeptest} with (pointwise) asymptotic level $\alpha$ (at $P_0$), we have
\begin{equation}
\label{det2}
\limsup_{n \to \infty } \inf_{P \in \mathcal P_*(\{1\})} \Exp_{P^{\otimes n}}[\varphi_n] \le \alpha,
\end{equation}
that is, for any $\eps>0$, there exists a sequence of alternatives $P_n$ with $\xi(P_n)=1$ such that $\Exp_{P_n^{\otimes n}}[\varphi_n] \le \alpha + \eps$ for all $n$.

\item The continuity assumption on $T_n$ is met for any continuous transformation $T_n = h_n(\hat \xi_n)$ of Chatterjee's coefficient $\hat \xi_n$. 
In particular, Lemma~\ref{lem:indeptest} applies to Chatterjee's asymptotic test $\varphi_{\xi, n} = \bm 1(\sqrt {5n/2} |\hat \xi_n| \ge u_{1-\alpha/2}) = \bm 1(\sqrt {5n/2} |\hat \xi_n| - u_{1-\alpha/2} \ge 0 )$. 
Moreover, it also applies to the test $\varphi_{\xi, n,M} = \bm 1(\sqrt {5n/2} |\hat \xi_{n,M}| \ge u_{1-\alpha/2}) $,
where $\hat \xi_{n,M}$ denotes the boosted version of Chatterjee's rank coefficient from \cite{linfan2022}; note that the continuity assumption is then a consequence of Remark 2 in their paper. As these test have pointwise asymptotic level $\alpha$, \eqref{det2} holds for $\varphi_n= \varphi_{\xi,n}$ and $\varphi_n=\varphi_{\xi,n,M}$.
    
\item \cite{ShiDrtHan22}, see also \cite{CaoBic20} and \cite{bickel2022}, provide a local power analysis for Chatterjee's asymptotic test $\varphi_{\xi,n} =  \bm 1(\sqrt {5n/2} |\hat \xi_n| \ge u_{1-\alpha/2})$. They find that, along specific local alternatives converging to the null hypothesis of independence, the tests are not rate optimal: they have trivial asymptotic power $\alpha$ against alternatives converging at rate $n^{-1/2}$. Our results complement their results, with the important difference that our alternatives are `maximally separated' from the null hypothesis as measured by $\xi$ (in that sense, they may be considered `non-local alternatives', despite that they are local with respect to the weak topology on the set of bivariate probability measures).

\end{compactenum}
\end{remark}

\section{Consequences for confidence intervals}
\label{sec:confint}

In the product experiment from the previous section, with $\mathcal P_*\subset \mathcal P$ a set of candidate models for the distribution of $(X,Y)$ that includes all distributions with continuous margins, let $C_n=[T_n^-, T_n^+]$ be a function from $(\R^2)^n$ to the set of closed intervals in $[0,1]$, where $T_n^- \le T_n^+$ are measurable. We either assume that $C_n$ is a confidence interval for $\xi(P)$ of level $1-\alpha$, that is,
\begin{align} \label{eq:ci-null}
P^{\otimes n}( \xi(P) \in C_n) \ge 1-\alpha \quad \forall P \in \mathcal P_*,
\end{align}
or that $C_n$ is a uniform asymptotic confidence interval for $\xi(P)$ of level $1-\alpha$, that is, 
\begin{align} \label{eq:ci-null-asy}
\liminf_{n \to \infty} 
\inf_{P \in \mathcal P_*} 
P^{\otimes n}( \xi(P) \in C_n) \ge 1-\alpha.
\end{align}
The following result is inspired by similar results in \cite{Duf97, Pfa98}.

\begin{lemma}
\label{lem:confint}
Fix $\alpha \in (0,1)$ and 
suppose that $T_n^\pm$ are continuous on the set of all sample vectors in $(\R^2)^n$ for which there are no ties in any marginal sample. 
Then \eqref{eq:ci-null} implies that, for any $P$ with continuous margins,
\begin{align} \label{eq:ci-bad1}
P^{\otimes n}( \xi_0 \in C_n) \ge 1- \alpha \quad \forall \xi_0 \in [\xi(P),1]
\end{align}
and 
\begin{align} \label{eq:ci-bad2}
P^{\otimes n}([\xi(P), 1] \subset C_n ) \ge 1-2 \alpha.
\end{align} 
Likewise, \eqref{eq:ci-null-asy} implies that, for any $P$ with continuous margins,
\begin{align} \label{eq:ci-bad1-asy}
\liminf_{n \to \infty} P^{\otimes n}( \xi_0 \in C_n) \ge 1- \alpha \quad \forall \xi_0 \in [ \xi(P),1]
\end{align}
and 
\begin{align} \label{eq:ci-bad2-asy}
\liminf_{n \to \infty} P^{\otimes n}( [\xi(P), 1] \subset C_n ) \ge 1-2 \alpha.
\end{align} 
\end{lemma}

\begin{proof}
Fix $P$ with continuous margins and $\xi_0 \in [\xi(P),1]$ . We start by assuming \eqref{eq:ci-null}. Choose a sequence $P_k$ weakly converging to $P$ with continuous margins such that $\xi(P_k)=\xi_0$; this is possible by Proposition~\ref{prop:Xi}. Let $T_n$ denote the random vector $(T_n^-, T_n^+)$.
In view of the continuity assumption on $T_n^\pm$, we have $P_k^{\otimes n} \circ T_n^{-1} \weak P^{\otimes n} \circ T_n^{-1}$ for $k \to \infty$, by the continuous mapping theorem. Define $I(\xi_0) = \{(x,y) \in \R^2: x \le \xi_0 \le y \}$, which is a closed set in $\R^2$. Then, by the Portmanteau theorem, 
\begin{align} \label{eq:proof-ci1}
P^{\otimes n} (\xi_0 \in C_n) 
= 
P^{\otimes n} (T_{n}^- \le \xi_0 \le T_n^+)
&= \nonumber
P^{\otimes n} ( T_n \in I(\xi_0) )
\\&\ge \nonumber 
\lim_{k \to \infty} P_k^{\otimes n} ( T_n \in I(\xi_0) )
\\&=
\lim_{k \to \infty} P_k^{\otimes n} ( \xi(P_k) \in C_n ) \ge 1-\alpha,
\end{align}
where we used that $\xi(P_k)=\xi_0$ and \eqref{eq:ci-null}. This implies \eqref{eq:ci-bad1}.

Applying \eqref{eq:ci-bad1} with $\xi_0=\xi(P)$ and $\xi_0=1$ and invoking the union bound, we obtain that
\begin{align} \label{eq:proof-ci2}
P^{\otimes n} ( \{\xi(P),1\} \subset C_n) 
&= \nonumber
1- P^{\otimes n}(\xi(P) \notin C_n \text{ or } 1 \notin C_n)
\\&\ge  \nonumber
1 - P^{\otimes n} ( \xi(P) \notin C_n) - P^{\otimes n} ( 1 \notin C_n) 
\\&= 
P^{\otimes n} ( \xi(P) \in C_n) - P^{\otimes n} ( 1 \in C_n)-1 
\ge 1-2 \alpha.
\end{align}
Since $C_n$ is interval-valued, the event $ \{\xi(P),1\} \subset C_n$ is equivalent to the event $[\xi(P),1] \subset C_n$, which implies \eqref{eq:ci-bad2}.

Next, suppose that \eqref{eq:ci-null-asy} is met. The inequality chain in \eqref{eq:proof-ci1} implies
\[
P^{\otimes n} (\xi_0 \in C_n) 
\ge 
\lim_{k \to \infty} P_k^{\otimes n} ( \xi(P_k) \in C_n )
\ge 
\inf_{P \in \mathcal P_*} P^{\otimes n}( \xi(P) \in C_n).
\]
Taking the liminf on both sides implies \eqref{eq:ci-bad1-asy}. Likewise, \eqref{eq:ci-bad2-asy} follows from \eqref{eq:proof-ci2}.
\end{proof}

\begin{remark}~
\begin{compactenum}[(i)]
\item
As in the previous section, the continuity assumption on $T_n^\pm$ would for instance be met if $T_n^\pm = \hat \xi_n + u_n^\pm$, with Chatterjee's empirical coefficient $\hat \xi_n$ and some deterministic sequences $u_n^\pm$.

\item The first part of Lemma~\ref{lem:confint} essentially shows that meaningful confidence sets of exact level $\alpha$ cannot exist: for all $P$ with continuous margins, they must be unreasonably large with probability at least $1-2\alpha$. Likewise, the second part of the lemma shows that meaningful uniform asymptotic confidence intervals do not exist. 

However, we want to stress that `uniformity' in the asymptotic version is essential; there may  exist meaningful confidence intervals of (pointwise) asymptotic level $1-\alpha$ (which is in fact the more common notion of an asymptotic confidence interval). Such intervals may possibly be constructed by asymptotic theory \citep{LinHan22} or by subsampling \citep{DetKro24}.

\item Qualitatively similar results hold if the candidate set $\mathcal P_*$ is only required to include all distributions $P$ with continuous margins that satisfy $\xi(P) \le \kappa$, where $\kappa\in(0,1)$ is some fixed constant. The special case $\xi(P)=1$ is excluded from the condition in \eqref{eq:ci-null} and \eqref{eq:ci-null-asy} then, which seems natural in view of the fact that asymptotic normality of Chatterjee's sample correlation coefficient has only been shown for $\xi(P)<1$ \citep{LinHan22}. In that case, a careful look at the proof of Lemma~\ref{lem:confint} shows that the interval $[\xi(P),1]$ in \eqref{eq:ci-bad1}--\eqref{eq:ci-bad2-asy}  must be replaced by $[\xi(P), \kappa]$. This interval is still unreasonably large, and does not shrink for increasing sample size.
\end{compactenum}
\end{remark}

\section*{Acknowledgements} 
%The authors are grateful to three unknown referees and an associate editor for their constructive comments on an earlier version of this paper.
Financial support by the Deutsche Forschungsgemeinschaft (DFG, German Research Foundation; Project-ID 520388526;  TRR 391:  Spatio-temporal Statistics for the Transition of Energy and Transport) is gratefully acknowledged.

\bibliographystyle{apalike}
\bibliography{biblio}

\end{document}